\documentclass[10pt]{article}

\usepackage[inner=2cm,outer=2cm]{geometry} 
\usepackage{graphicx}          
\usepackage{epstopdf}
\usepackage{amsthm}
\usepackage{amssymb, amsmath}
\usepackage{subcaption}
\usepackage{bbm}
\usepackage{color}
\usepackage{cite}

\newtheorem{thm}{\bf{Theorem}}[section]
\newtheorem{rem}{\bf{Remark}}[section]
\newtheorem{assum}{\bf{Assumption}}[section]
\newtheorem{cor}{\bf{Corollary}}[section]

\newcommand{\nn}{\nonumber}

\newcommand{\vMz}{\mathbf{M}_0}

\newcommand{\vmm}{\mathbf{m}}
\newcommand{\vmt}{\widetilde{\mathbf{m}}}
\newcommand{\vxx}{\mathbf{x}}
\newcommand{\vww}{\mathbf{w}}
\newcommand{\vyy}{\mathbf{y}}
\newcommand{\vuu}{\mathbf{u}}
\newcommand{\ver}{\mathbf{e}}
\newcommand{\vvv}{\mathbf{v}}

\newcommand{\vb}{\mathbf{b}}

\DeclareMathOperator{\tr}{tr}

\let\turc\c
\usepackage{etoolbox}
\robustify\turc

\allowdisplaybreaks

\begin{document}
\sloppy

\title{Dynamic Signaling Games with Quadratic Criteria under Nash and Stackelberg Equilibria\footnote{This research was supported in part by the Natural Sciences and Engineering Research Council of Canada (NSERC), and the Scientific and Technological Research Council of Turkey (T\"{U}B\.{I}TAK). Part of this work was presented at the 2016 IEEE International Symposium on Information Theory (ISIT), Barcelona, Spain, 2016 \cite{isitDynamic}, and at the 2017 American Control Conference (ACC), Seattle, WA, 2017 \cite{acc2017}.}} 

\author{Serkan~Sar{\i}ta\c{s}$^1$ \and Serdar~Y\"uksel$^2$ \and Sinan~Gezici$^3$}
\date{%
	$^1$Division of Decision and Control Systems, KTH Royal Institute of Technology, SE-10044, Stockholm, Sweden. Email: saritas@kth.se.\\%
	$^2$Department of Mathematics and Statistics, Queen's University, K7L 3N6, Kingston, Ontario, Canada.  Email: yuksel@mast.queensu.ca.\\
	$^3$Department of Electrical and Electronics Engineering, Bilkent University, 06800, Ankara, Turkey. Email: gezici@ee.bilkent.edu.tr.\\[2ex]%
}


\maketitle



\begin{abstract}
This paper considers dynamic (multi-stage) signaling games involving an encoder and a decoder who have subjective models on the cost functions. We consider both Nash (simultaneous-move) and Stackelberg (leader-follower) equilibria of dynamic signaling games under quadratic criteria. For the multi-stage scalar cheap talk, we show that the final stage equilibrium is always quantized and under further conditions the equilibria for all time stages must be quantized. In contrast, the Stackelberg equilibria are always fully revealing. In the multi-stage signaling game where the transmission of a Gauss-Markov source over a memoryless Gaussian channel is considered, affine policies constitute an invariant subspace under best response maps for Nash equilibria; whereas the Stackelberg equilibria always admit linear policies for scalar sources but such policies may be non-linear for multi-dimensional sources. We obtain an explicit recursion for optimal linear encoding policies for multi-dimensional sources, and derive conditions under which Stackelberg equilibria are informative.
\end{abstract}


\section{Introduction}

Signaling games and cheap talk are concerned with a class of Bayesian games where an informed player (encoder or sender) transmits information to another player (decoder or receiver). In these problems, the objective functions of the players are not aligned unlike the ones in the classical communication problems. The single-stage cheap talk problem was studied by Crawford and Sobel \cite{SignalingGames}, who obtained the surprising result that under some technical conditions on the cost functions, the cheap talk problem only admits equilibria that involve quantized encoding policies. This is in contrast to the usual communication/information theoretic case where the goals are aligned. In this paper, we consider multi-stage signaling games. The details are presented in the following.

A single-stage (static) {\it cheap talk} problem can be formulated as follows: An informed player (encoder) knows the value of the $\mathbb{M}$-valued random variable $M$ and transmits the $\mathbb{X}$-valued random variable $X$ to another player (decoder), who generates his $\mathbb{M}$-valued optimal decision $U$ upon receiving $X$. The policies of the encoder and decoder are assumed to be deterministic; i.e., $x=\gamma^e(m)$ and $u=\gamma^d(x)=\gamma^d(\gamma^e(m))$. Let $c^e(m,u)$ and $c^d(m,u)$ denote the cost functions of the encoder and the decoder, respectively, when the action $u$ is taken for the corresponding message $m$. Then, given the encoding and decoding policies, the encoder's induced expected cost is $J^e\left(\gamma^e,\gamma^d\right) = \mathbb{E}\left[c^e(m, u)\right]$, whereas, the decoder's induced expected cost is $J^d\left(\gamma^e,\gamma^d\right) = \mathbb{E}\left[c^d(m, u)\right]$.
If the transmitted signal $x$ is also an explicit part of the cost functions $c^e$ and/or $c^d$, then the communication between the players is not costless and the formulation turns into a {\it signaling game} problem. Such problems are studied under the tools and concepts provided by {\it game theory} since the goals are not aligned. Although the encoder and decoder act sequentially in the game as described above, how and when the decisions are made and the nature of the commitments to the announced policies significantly affect the analysis of the equilibrium structure. Here, two different types of equilibria are investigated: the {\it Nash equilibrium}, in which the encoder and the decoder make simultaneous decisions, and the {\it Stackelberg equilibrium}, in which the encoder and the decoder make sequential decisions where the encoder is the leader and the decoder is the follower\footnote{\label{note1} For the Nash equilibrium, the encoder and the decoder take actions sequentially but do not announce their policies beforehand (i.e., they announce simultaneously) while they do so for the Stackelberg equilibrium.}. In this paper, the terms {\it Nash game} and the {\it simultaneous-move game}\footnote{Note that throughout the manuscript, {\it simultaneous-move} refers to {\it simultaneous-announcement}, see\textsuperscript{ $1$}.} will be used interchangeably, and similarly, the  {\it Stackelberg game} and the {\it leader-follower game} will be used interchangeably. 

In the simultaneous-move game, the encoder and the decoder announce their policies at the same time, and a pair of policies $(\gamma^{*,e}, \gamma^{*,d})$ is said to be a {\bf Nash equilibrium} \cite{basols99} if
\begin{align}
\begin{split}
J^e(\gamma^{*,e}, \gamma^{*,d}) &\leq J^e(\gamma^{e}, \gamma^{*,d}) \quad \forall \gamma^e \in \Gamma^e \,,\\
J^d(\gamma^{*,e}, \gamma^{*,d}) &\leq J^d(\gamma^{*,e}, \gamma^{d}) \quad \forall \gamma^d \in \Gamma^d \,,
\label{eq:nashEquilibrium}
\end{split}
\end{align}
where $\Gamma^e$ and $\Gamma^d$ are the sets of all deterministic (and Borel measurable) functions from $\mathbb{M}$ to $\mathbb{X}$ and from $\mathbb{X}$ to $\mathbb{M}$, respectively. As observed from the definition \eqref{eq:nashEquilibrium}, under the Nash equilibrium, each individual player chooses an optimal strategy given the strategies chosen by the other players.

On the other hand, in a leader-follower game, the leader (encoder) commits to and announces his optimal policy before the follower (decoder) does, the follower observes what the leader is committed to before choosing and announcing his optimal policy, and a pair of policies $(\gamma^{*,e}, \gamma^{*,d})$ is a {\bf Stackelberg equilibrium} \cite{basols99} if
\begin{align}
\begin{split}
&J^e(\gamma^{*,e}, \gamma^{*,d}(\gamma^{*,e})) \leq J^e(\gamma^e, \gamma^{*,d}(\gamma^e)) \quad \forall \gamma^e \in \Gamma^e \,,\\
&\hspace{-0.25cm} \text{where } \gamma^{*,d}(\gamma^e) \text{ satisfies} \\
&J^d(\gamma^{e}, \gamma^{*,d}(\gamma^{e})) \leq J^d(\gamma^{e}, \gamma^d(\gamma^{e})) \quad \forall \gamma^d \in \Gamma^d  \,.
\label{eq:stackelbergEquilibrium}
\end{split}
\end{align}
As it can be seen from the definition \eqref{eq:stackelbergEquilibrium}, the decoder takes his optimal action $\gamma^{*,d}(\gamma^{e})$ after observing the policy of the encoder $\gamma^{e}$. In the Stackelberg game, the leader cannot backtrack on his commitment, but has a leadership role since he can manipulate the follower by anticipating follower's actions; Bayesian persuasion games have a similar spirit \cite{bayesianPersuasion}. We provide a further literature review on this later. 



If an equilibrium is achieved when $\gamma^{*,e}$ is non-informative (e.g., the transmitted message and the source are independent) and $\gamma^{*,d}$ uses only the prior information (since the received message is useless), then we call such an equilibrium a {\it non-informative (babbling) equilibrium}, which always exists for cheap talk \cite{SignalingGames}.

Heretofore, only \textit{single-stage game}s are considered. If a game is played over a number of time periods, the game is called a \textit{multi-stage game}. In this paper, with the term {\it dynamic}, we will refer to \textit{multi-stage game} setups, which also has been the usage in the prior literature \cite{dynamicStrategicInfoTrans}; even though strictly speaking a single stage setup may also be viewed to be dynamic  \cite{YukselBasarBook} since the information available to the decoder is totally determined by encoder's actions. 

Let $m_{[0,N-1]}=\{m_0, m_1, \ldots , m_{N-1}\}$ be a collection of random variables to be encoded sequentially (causally) to a decoder. At the $k$-th stage of an $N$-stage game, the encoder knows $\mathcal{I}_k^e = \{m_{[0,k]}, x_{[0,k-1]} \}$ with $\mathcal{I}_0^e = \{m_0 \}$, and transmits $x_k$ to the decoder who generates his optimal decision by knowing $\mathcal{I}_k^d = \{x_{[0,k]} \}$. Thus, under the policies considered, $x_k=\gamma_k^e(\mathcal{I}_k^e)$ and $u_k=\gamma_k^d(\mathcal{I}_k^d)$. The encoder's goal is to minimize
\begin{equation}
J^e\left(\gamma^e_{[0,N-1]},\gamma^d_{[0,N-1]}\right) = \mathbb{E}\left[\sum_{k=0}^{N-1}c_k^e(m_k, u_k) \right]\,,
\label{eq:dynamicEncCost}
\end{equation}
whereas the decoder's goal is to minimize
\begin{equation}
J^d\left(\gamma^e_{[0,N-1]},\gamma^d_{[0,N-1]}\right) = \mathbb{E}\left[\sum_{k=0}^{N-1}c_k^d(m_k, u_k)\right]
\label{eq:dynamicDecCost}
\end{equation}
by finding the optimal policy sequences $\gamma^{*,e}_{[0,N-1]} = \{\gamma_0^{*,e}, \gamma_1^{*,e}, \cdots, \gamma_{N-1}^{*,e}\}$ and $\gamma^{*,d}_{[0,N-1]} = \{\gamma_0^{*,d}, \gamma_1^{*,d}, \cdots, \gamma_{N-1}^{*,d}\}$, respectively. Using the encoder cost in \eqref{eq:dynamicEncCost} and the decoder cost in \eqref{eq:dynamicDecCost}, the Nash equilibrium and the Stackelberg equilibrium for multi-stage games can be defined similarly as in \eqref{eq:nashEquilibrium} and \eqref{eq:stackelbergEquilibrium}, respectively. 

Under both equilibria concepts, we consider the setups where the decision makers act optimally for each history path of the game (available to each decision maker) and the updates are Bayesian; thus the equilibria are to be interpreted under a {\it perfect Bayesian equilibria} concept. Since we assume such a (perfect Bayesian) framework, the equilibria lead to sub-game perfection and each decision maker performs optimal Bayesian decisions for every realized play path. For example, more general Nash equilibrium scenarios such as non-credible threats or equilibria that are not strong time-consistent \cite[Definition 2.4.1]{YukselBasarBook} may not be considered. We also note that both Nash and Stackelberg equilibrium concepts find various applications depending on the assumptions on the leader, that is, the encoder, in view of the commitment conditions \cite{NashvsStackelberg}.

In this paper, the quadratic cost functions are assumed; i.e., $c_k^e(m_k, u_k)=(m_k-u_k-b)^2$ and $c_k^d(m_k, u_k)=(m_k-u_k)^2$ where $b$ is the bias term as in \cite{SignalingGames} and \cite{tacWorkArxiv}. 


\subsection{Literature Review}

As noted in \cite{tacWorkArxiv}, for team problems, although it is difficult to obtain optimal solutions under general information structures, it is apparent that more information provided to any of the decision makers does not negatively affect the utility of the players. There is also a well-defined partial order of information structures as studied by Blackwell \cite{Blackwell} and \cite{YukselBasarBook}. However, for general zero-sum or non-zero-sum game problems, informational aspects are very challenging to address; more information can have negative effects on some or even all of the players in a system, see e.g. \cite{hirshleifer1971private}. 

The cheap talk and signaling game problems are applicable in networked control systems when a communication channel exists among competitive and non-cooperative decision makers. For example, in a smart grid application, there may be strategic sensors in the system \cite{misBehavingAgents} that wish to change the equilibrium for their own interests through reporting incorrect measurement values, e.g., to enhance prolonged use in the system. 
For further applications, see \cite{misBehavingAgents,csLloydMax}. All of these applications lead to a new framework where the value of information and its utilization have a drastic impact on the system under consideration.

The reader is referred to \cite{tacWorkArxiv} for further discussion and references on single-stage signaling games. On the multi-stage side, much of the literature has focused on Stackelberg equilibria as we note below. A notable exception is \cite{dynamicStrategicInfoTrans}, where the multi-stage extension of the setup of Crawford and Sobel is analyzed for a source which is a fixed random variable distributed according to some density on $[0,1]$ (see Theorem~\ref{rem:golosovExplanation} for a detailed discussion on this very relevant paper). \cite{frenchGroup} considers the information design problem (which is originated from the Bayesian persuasion game \cite{bayesianPersuasion}) between an encoder and a decoder with non-aligned utility functions under the Stackelberg equilibrium. For the case in which the non-alignment between the cost functions of the encoder and the decoder (i.e., the bias term $b$) is a function of a Gaussian random variable (r.v.) correlated with the Gaussian source and secret to the decoder (contrarily to the original case in which it is fixed and known to the decoder \cite{SignalingGames}, which is also studied in \cite{tacWorkArxiv} and in this manuscript)\footnote{Since we assume a fixed and public $b$ in contrast to a private and random $b$ which is correlated with the source as in \cite{CedricWork, akyolITapproachGame,omerHierarchial}, the results obtained in the former setup cannot be applied directly to the latter one; i.e., the Stackelberg equilibria of these two setups are different.}, the Stackelberg equilibrium is investigated in \cite{CedricWork, akyolITapproachGame,omerHierarchial}. It is shown that the best response of the transmitter is affine by restricting receiver strategies to be affine when the communication is noiseless \cite{CedricWork}, whereas the optimality of linear sender strategies is proved within the general class of policies even with additive Gaussian noise channels \cite{akyolITapproachGame}. The multi-stage Gaussian signaling game under general quadratic cost functions is studied in \cite{omerHierarchial} and it is shown that linear encoder and decoder strategies can achieve the Stackelberg equilibrium under a finite horizon when the private state of the encoder is a controlled Gauss-Markov process. \cite{VasalAnastasopoulos} investigates the Nash equilibrium of a multi-stage linear quadratic Gaussian game with asymmetric information, and it is shown that under certain conditions, players' strategies are linear in their private types. The cheap talk with finite state and action spaces and multiple round of pre-play communication is investigated where both the encoder and decoder take costly actions at the end of the pre-play communication in \cite{longCheapTalk}, and it is proved that the multiple round of communication improves information revelation. The dynamic extension to the optimal information disclosure \cite{bayesianPersuasion} with a finite state space and a finite number of periods is considered in \cite{suspenseSurprise}, in which the sender commits to a policy similar to the Stackelberg case. In our earlier work \cite{tacWorkArxiv}, we considered both Nash equilibria and Stackelberg equilibria. In this paper, we build on \cite{tacWorkArxiv}, and extend the analysis to the multi-stage case. 


\subsection{Contributions} 

\begin{itemize}
	\item[(i)] We show that in the multi-stage cheap talk game under Nash equilibria, the last stage equilibria are quantized for scalar i.i.d. and Markov sources with arbitrary conditional probability measures and not fully revealing for multi-dimensional sources, whereas the equilibrium must be fully revealing in the multi-stage cheap talk game under Stackelberg equilibria for both scalar and multi-dimensional sources. Further, for scalar i.i.d. sources, the quantized nature of the Nash equilibrium for all stages is established under mild conditions. 
	\item[(ii)] For the multi-stage signaling game under Nash equilibria, it is shown that affine encoder and decoder policies constitute an invariant subspace under best response dynamics. We provide conditions for the existence of informative Stackelberg equilibria for scalar Gauss-Markov sources and scalar Gaussian channels where we also show that Stackelberg equilibria are always linear for scalar sources and channels, which is not always the case for multi-dimensional setups. For multi-dimensional setups, a dynamic programming formulation is presented for Stackelberg equilibria when the encoders are linear.
\end{itemize}

\section{Multi-Stage Cheap Talk}

For the purpose of illustration, the system model of the $2$-stage cheap talk is depicted in Fig.~\ref{figure:allScheme}-(\subref{figure:noiselessScheme}). Before presenting the technical results related to the multi-stage cheap talk, we provide the result on the static (single-stage) cheap talk with deterministic encoders from our previous study, which will be utilized in several results later on this paper.

\begin{figure}
	\centering
	\null\hfill
	\begin{subfigure}[t]{0.48\textwidth}
		\centering
		\includegraphics[scale=0.54]{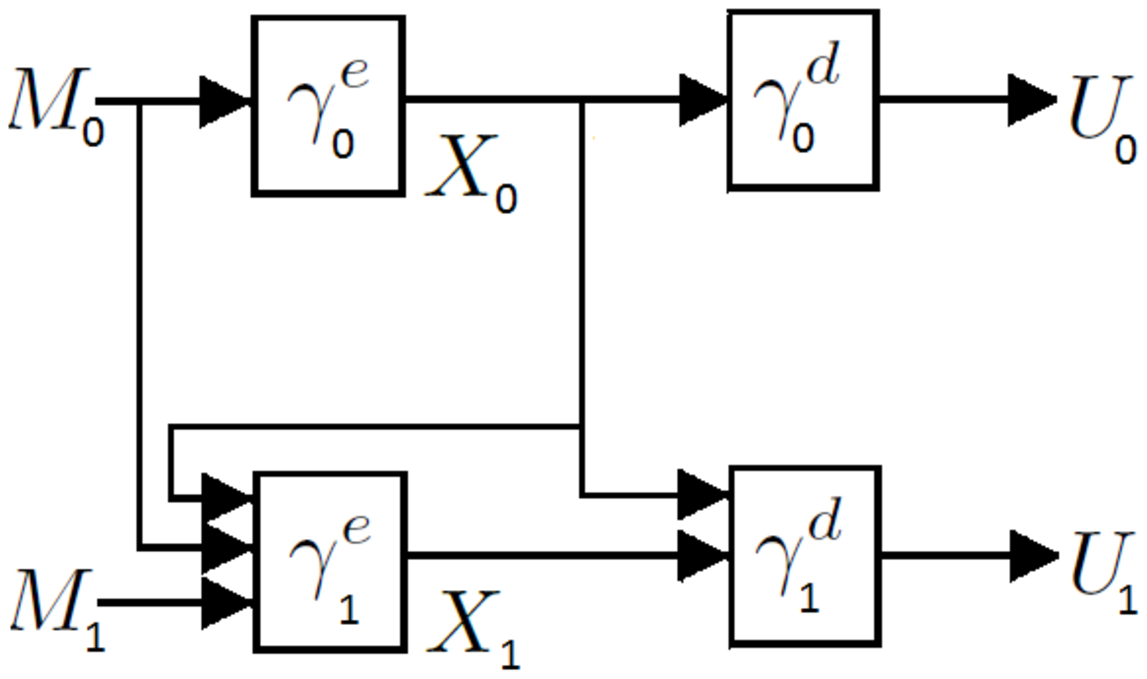}
		\caption{$2$-stage cheap talk.}
		\label{figure:noiselessScheme}
	\end{subfigure}
	\hfill
	\begin{subfigure}[t]{0.48\textwidth}
		\centering
		\includegraphics[scale=0.58]{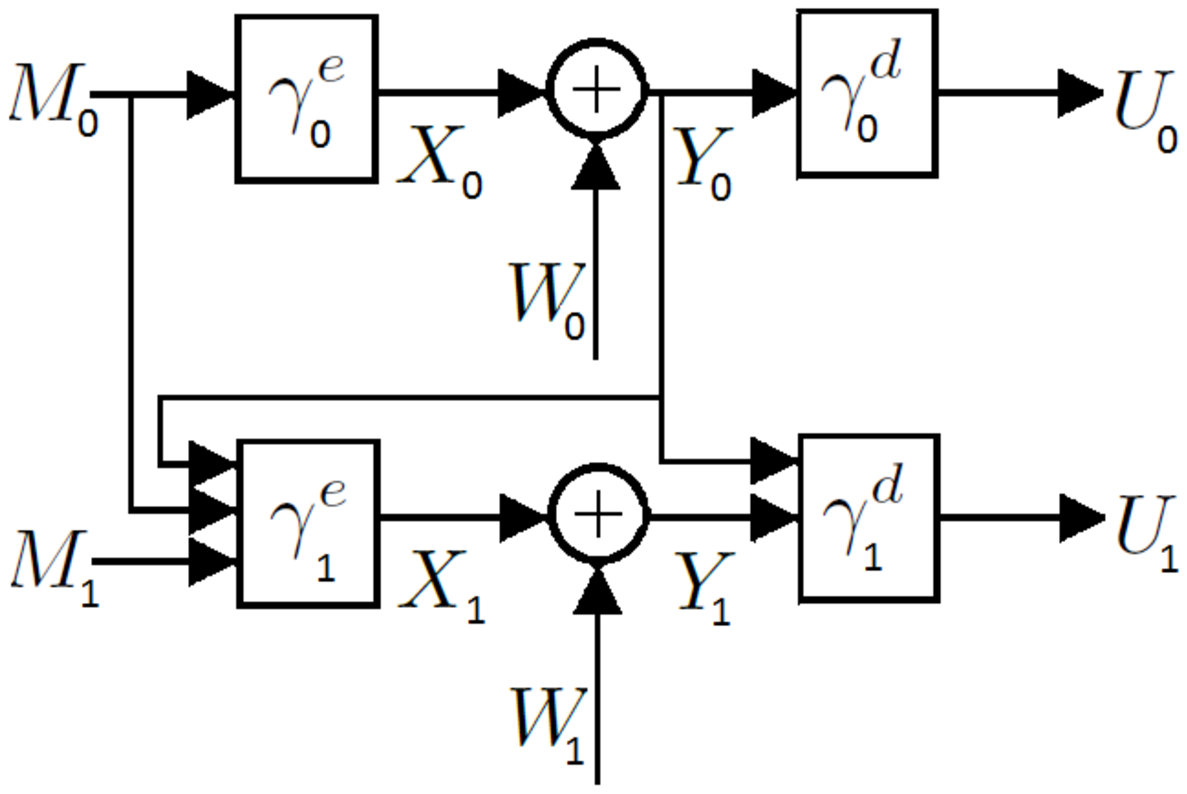}
		\caption{$2$-stage signaling game.}
		\label{figure:genScheme}
	\end{subfigure}
	\hfill\null
	\caption{Signaling game models.}
	\label{figure:allScheme}
\end{figure}

\begin{thm} \cite[Theorem 3.2]{tacWorkArxiv} \label{noiselessDiscrete}
	Let the strategy set of the encoder consists of the set of all measurable (deterministic) functions from $\mathbb{M}$ to $\mathbb{X}$. Then, an equilibrium encoder policy has to be quantized almost surely, that is, it is equivalent to a quantized policy for the encoder in the sense that the performance of any equilibrium encoder policy is equivalent to the performance of a quantized encoder policy. Furthermore, the quantization bins are convex.
\end{thm}

To facilitate our analysis to handle certain intricacies that arise due to the multi-stage setup in this paper, in the following, we state that the result in Theorem~\ref{noiselessDiscrete} also holds when the encoder is allowed to adapt randomized encoding policies by extending \cite[Lemma 1]{SignalingGames} as follows:

\begin{thm} \label{thm:randCheapTalk}
	The conclusion of Theorem~\ref{noiselessDiscrete} holds if the policy space of the encoder is extended to the set of all stochastic kernels from $\mathbb{M}$ to $\mathbb{X}$ for any arbitrary source.\footnote{$P$ is a stochastic kernel from $\mathbb{M}$ to $\mathbb{X}$ if $P(\cdot|m)$ is a probability measure on ${\mathcal B}(\mathbb{X})$ for every $m\in \mathbb{M}$, and $P(A|\cdot)$ is a Borel measurable function of $m$ for every $A\in{\mathcal B}(\mathbb{X})$.} That is, even when the encoder is allowed to use private randomization, all equilibria are equivalent to those that are attained by quantized equilibria. 
\end{thm}
\begin{proof}
\cite[Lemma 1]{SignalingGames} proves that all equilibria with a randomized encoder have finitely many partitions when the source has bounded support. Theorem~\ref{noiselessDiscrete} extends this result to a countable number of partitions (i.e., distinct decoder actions must differ by at least $2|b|$) for any source with an arbitrary probability measure, but assuming a deterministic encoder. The proof methods in both of \cite[Lemma 1]{SignalingGames} and Theorem~\ref{noiselessDiscrete} can be combined, which implies that equilibria must be quantized.
\end{proof}

Theorem~\ref{thm:randCheapTalk} will be used crucially in the following analysis; since in a multi-stage game, at a given time stage, the source variables from the earlier stages can serve as private randomness for the encoder. As a prelude to the more general Markov source setup, we first analyze the multi-stage cheap talk game with an i.i.d. scalar source.

\subsection{Multi-Stage Game with an i.i.d. Source}

\begin{thm}\label{thm:encoderNQuantized}
	In the $N$-stage repeated cheap talk game, the equilibrium policies for the final stage encoder must be quantized for any collection of policies $\left(\gamma^e_{[0,N-2]}, \gamma^d_{[0,N-2]} \right)$ and for any real-valued source model with arbitrary probability measure $P(\mathrm{d}m_{N-1})$.	
\end{thm}
\begin{proof}
Here, we prove the result for the $2$-stage setup, the extension to multiple stages is merely technical, as we comment on at the end of the proof. Let $c_1^e(m_1,u_1)$ be the second stage cost function of the encoder. Then the expected cost of the second stage encoder $J_1^e$ is
\begin{align}
J_1^e&=\int P(\mathrm{d}m_0,\mathrm{d}m_1,\mathrm{d}x_0,\mathrm{d}x_1)\,c_1^e(m_1,u_1) \nn\\
\begin{split}
&= \int P(\mathrm{d}x_0)\, \int P(\mathrm{d}m_1|x_0)\,P(\mathrm{d}m_0|m_1,x_0)\\ 
&\qquad\qquad\times c_1^e(m_1,\gamma_1^d(x_0,\gamma_1^e(m_0,m_1,x_0))) \;. 
\label{eq:x0paramertrize}
\end{split} 
\end{align} 
Since $P(\mathrm{d}m_1|x_0)=P(\mathrm{d}m_1)$ and $P(\mathrm{d}m_0|m_1,x_0)=P(\mathrm{d}m_0|x_0)$ for an i.i.d. source, the inner integral of \eqref{eq:x0paramertrize} can be considered as an expression for a given $x_0$. Thus, given the second stage encoder and decoder policies $\gamma_1^e(m_0,m_1,x_0)$ and $\gamma_1^d(x_0,x_1)$, it is possible to define policies which are parametrized by the common information $x_0$ almost surely so that $\widehat{\gamma}_{x_0}^e(m_0,m_1) \triangleq \gamma_1^e(m_0,m_1,x_0)$ and $\widehat{\gamma}_{x_0}^d(x_1) \triangleq \gamma_1^d(x_0,x_1)$. 

Now fix the first stage policies $\gamma_0^e$ and $\gamma_0^d$. Suppose that the second stage encoder does not use $m_0$; i.e., $\widehat{\gamma}_{x_0}^{e\prime}(m_1)$ is the policy of the second stage encoder. For the policies $\widehat{\gamma}_{x_0}^{e\prime}(m_1)$ and $\widehat{\gamma}_{x_0}^d(x_1)$, by using the second stage encoder cost function $H_{x_0}(m_1,u_1)\triangleq \mathbb{E}[(m_1-u_1-b)^2|x_0]$ and the bin arguments from Theorem~\ref{noiselessDiscrete}, it can be deduced that, due to the continuity of $H_{x_0}(m_1,u_1)$ in $m_1$, the equilibrium policies for the second stage must be quantized for any collection of policies $(\gamma_0^e, \gamma_0^d)$ and for any given $x_0$. Now let the second stage encoder use $m_0$; i.e., $\widehat{\gamma}_{x_0}^e(m_0,m_1)$ is the deterministic policy of the second stage encoder, which can be regarded as an equivalent randomized encoder policy (as a stochastic kernel from $\mathbb{M}_1$ to $\mathbb{X}_1$) where $m_0$ is a real valued random variable independent of the source, $m_1$. From Theorem~\ref{thm:randCheapTalk}, the equilibrium is achievable with an encoder policy which uses only $m_1$; i.e., $\widehat{\gamma}_{x_0}^{e*}(m_1)$ is an encoder policy at the equilibrium and thus the equilibria are quantized.$\newline$
For the $N$-stage game, the common information of the final stage encoder and decoder becomes $x_{[0,N-2]}$, and $m_{[0,N-2]}$ is a vector valued random variable independent of the final stage source $m_{N-1}$.
\end{proof}

\begin{assum}\label{assumption:finiteEq}
	The source $m_k$ is so that the single-stage cheap-talk game satisfies the following:\newline
	(i) There exists a finite upper bound on the number of quantization bins that any equilibrium admits. \newline
	(ii) There exist finitely many equilibria corresponding to a given number of quantization bins.
\end{assum}

A number of comments on  Assumption \ref{assumption:finiteEq} is in order: This assumption is not unrealistic, e.g., all sources with bounded support, and sources with sufficiently light tail such as the exponential distribution satisfies this property, provided that their associated densities are one-sided and the sign of $b$ is negative \cite{isitNumberBinsArxiv}. A sufficient condition for Assumption \ref{assumption:finiteEq} is that the source admits a bounded support (which would require by \cite[Lemma 1]{SignalingGames} that there exists an upper bound on the number of bins in any equilibrium), and that a monotonicity condition \cite[conditions (M), or equivalently (M$^\prime$)]{SignalingGames} holds, which characterize the behavior of equilibrium policies. Note though that this condition is much more than what is needed in Assumption \ref{assumption:finiteEq}, since it entails the uniqueness of equilibria for a given number of bins: The uniqueness of equilibria even for team problems with $b=0$ requires restrictive $\log$-concavity conditions \cite[p. 1475]{Kieffer93}, \cite{Fleischer64}.

\begin{thm}	\label{thm:encoderAllQuantized}
	Under Assumption~\ref{assumption:finiteEq}, all stages must have quantized equilibria with finitely many bins in the $N$-stage repeated cheap talk game.
\end{thm}
\begin{proof}
Consider first the $2$-stage setup; i.e., given that the second stage has a quantized equilibrium by Theorem~\ref{thm:encoderNQuantized}, the quantized nature of the first stage will be established. Let $F(m_0,x_0)$ be a cost function for the first stage encoder if it encodes message $m_0$ as $x_0$. Since the second stage equilibrium cost does not depend on $m_0$ (since $m_0$ is a random variable independent of the source $m_1$ as shown in Theorem~\ref{thm:encoderNQuantized}), $F(m_0,x_0)$ can be written as $F(m_0,x_0) = \left(m_0-\gamma_0^d(x_0)-b\right)^2 + G(x_0)$ where $G(x_0)\triangleq \mathbb{E}_{m_1}\left[\left(m_1-\gamma_1^{*,d}(x_0,\gamma_1^{*,e}(m_1,x_0))-b\right)^2\Big| x_0\right]$ is the expected cost of the second stage encoder, and $\gamma_1^{*,e}$ and $\gamma_1^{*,d}$ are the second stage encoder and decoder policies at the equilibrium, respectively. Note that the second stage encoder cost can $G(x_0)$ take finitely many different values by Assumption~\ref{assumption:finiteEq}-(ii). Now define the equivalence classes $T_{x_p}$ for every $x_p\in\mathbb{X}$ as $T_{x_p}=\{x\in\mathbb{X} : G(x)=G(x_p)\}$; i.e., the equivalence classes $T_{x_p}$ keep the first stage encoder actions that result in the same second stage cost in the same set. Note that there are finitely many equivalence classes $T_{x_p}$ since $G(x_0)$ can take finitely many different values. 


If the number of bins of the first stage equilibrium is less than or equal to the number of the equivalence classes $T_{x_p}$, then the proof is complete; i.e., the first stage equilibrium is already quantized with finitely may bins. Otherwise,
one of the equivalence classes $T_{x_p}$ has at least two elements; say $x_0^{\alpha}$ and $x_0^{\beta}$, which implies $G(x_0^{\alpha})=G(x_0^{\beta})$. Let corresponding bins of the actions $x_0^{\alpha}$ and $x_0^{\beta}$ be $\mathcal{B}_0^{\alpha}$ and $\mathcal{B}_0^{\beta}$, respectively. Also let $m_0^{\alpha}$ and $m_0^{\beta}$ represent any point in $\mathcal{B}_0^{\alpha}$ and $\mathcal{B}_0^{\beta}$, respectively; i.e., $m_0^{\alpha} \in \mathcal{B}_0^{\alpha}$ and $m_0^{\beta} \in \mathcal{B}_0^{\beta}$. The decoder chooses an action $u_0^{\alpha}=\gamma_0^d(x_0^{\alpha})$ when the encoder sends $x_0^{\alpha}=\gamma_0^e(m_0^{\alpha})$, and an action $u_0^{\beta}=\gamma_0^d(x_0^{\beta})$ when the encoder sends $x_0^{\beta}=\gamma_0^e(m_0^{\beta})$ in order to minimize his total cost; further, we can assume that $u_0^{\alpha}<u_0^{\beta}$ without loss of generality. Due to the equilibrium definitions from the view of the encoder, $F(m_0^{\alpha},x_0^{\alpha})<F(m_0^{\alpha},x_0^{\beta})$ and $F(m_0^{\beta},x_0^{\beta})<F(m_0^{\beta},x_0^{\alpha})$. These inequalities imply that 
\begin{align*}
\begin{split}
&(m_0^{\alpha}-u_0^{\alpha}-b)^2+G(x_0^{\alpha})<(m_0^{\alpha}-u_0^{\beta}-b)^2+G(x_0^{\beta}) \\
&\Rightarrow (u_0^{\alpha}-u_0^{\beta})(u_0^{\alpha}+u_0^{\beta}-2(m_0^{\alpha}-b)) < 0 \;,\\ &(m_0^{\beta}-u_0^{\beta}-b)^2+G(x_0^{\beta})<(m_0^{\beta}-u_0^{\alpha}-b)^2+G(x_0^{\alpha})\\
&\Rightarrow (u_0^{\beta}-u_0^{\alpha})(u_0^{\alpha}+u_0^{\beta}-2(m_0^{\beta}-b)) < 0 \;.\\
\end{split}
\end{align*} 
Thus, we have $u_0^{\alpha}+u_0^{\beta}-2(m_0^{\alpha}-b)>0$ and $u_0^{\alpha}+u_0^{\beta}-2(m_0^{\beta}-b)<0$, that reduce to \mbox{$m_0^{\alpha}<{u_0^{\alpha}+u_0^{\beta}\over2}+b<m_0^{\beta}$}.
Since $u^{\alpha}=\mathbb{E}[m|m \in \mathcal{B}^{\alpha}]$ and $u^{\beta}=\mathbb{E}[m|m \in \mathcal{B}^{\beta}]$ at the equilibrium, \mbox{$u_0^{\alpha}<{u_0^{\alpha}+u_0^{\beta}\over2}+b<u_0^{\beta} \Rightarrow u_0^{\beta}-u_0^{\alpha}>2|b|$} is obtained. Hence, there must be at least $2|b|$ distance between the actions of the first stage decoder which are in the same equivalence class. Therefore, the cardinality of any equivalence class $T_{x_p}$ is finite due to Assumption~\ref{assumption:finiteEq}-(i). Further, there are finitely many equivalence classes $T_{x_p}$ as shown above. These two results imply that the first stage equilibrium must be quantized with finitely many bins. Thus, due to Assumption~\ref{assumption:finiteEq}-(ii), there are finitely many equilibria in the first stage; i.e., the first stage encoder cost can take finitely many values.

For the $N$-stage game, we apply the similar recursion from the final stage to the first stage. It is already proven that the last two stage encoder cost can take finitely many values; thus, the same methods can be applied to show the quantized structure (with finitely many bins) of the equilibria for all stages. 
\end{proof}

\begin{rem}
	It is important to note that the first stage encoder minimizes his expected cost $J_0^e=\mathbb{E}[F(m_0,x_0)]$ by minimizing his cost $F(m_0,x_0)$ for every realizable $m_0$; this property will be later used as well.
\end{rem}

\subsection{Multi-Stage Game with a Markov Source: Nash Equilibria}

Here, the source $M_k$ is assumed to be real valued Markovian for $k\leq N-1$. The following result generalizes Theorem~\ref{thm:encoderNQuantized}, which only considered i.i.d. sources.
\begin{thm}\label{thm:encoderNMarkovReductionQuantized}
	In the $N$-stage cheap talk game with a Markov source, the equilibrium policies for the final stage encoder must be quantized for any collection of policies $\left(\gamma^e_{[0,N-2]}, \gamma^d_{[0,N-2]}\right)$  and for any real-valued source model with arbitrary probability measure. 	
\end{thm}
\begin{proof}
Here, we prove the results for the $2$-stage games as the extension is merely technical. Similar to that in Theorem~\ref{thm:encoderNQuantized}, the expected cost of the second stage encoder can be written as \eqref{eq:x0paramertrize}. After following similar arguments, the second stage encoder policy becomes $\widehat{\gamma}_{x_0}^e(m_0,m_1)\overset{(a)}{=}\widehat{\gamma}_{x_0}^e(g(m_1,r),m_1)=\widetilde{\gamma}_{x_0}^e(m_1,r)$ where (a) holds since any stochastic kernel from a complete, separable and metric space to another one, $P(\mathrm{d}m_0|m_1)$, can be realized by some measurable function $m_0=g(m_1,r)$ where $r$ is a $[0,1]$-valued independent random variable (see \cite[Lemma 1.2]{gihman2012controlled}, or in \cite[Lemma 3.1]{BorkarRealization}). Hence, the equilibria are quantized by Theorem~\ref{thm:encoderNQuantized}.
\end{proof}

As it can be observed from Theorem~\ref{thm:encoderAllQuantized}, to be able
to claim that the equilibria for all stages are quantized, we require very strong conditions. In fact, in the absence of such conditions, the equilibria for a Markov source can be quite counterintuitive and even fully revealing as we observe in the following theorem due to \cite{dynamicStrategicInfoTrans}.
\begin{thm}\cite{dynamicStrategicInfoTrans}\label{rem:golosovExplanation} 
	For a Markov source, there exist multi-stage cheap talk games with fully revealing equilibria.
\end{thm}
An example is presented in Golosov et. al. \cite{dynamicStrategicInfoTrans}, where an individual source is transmitted repeatedly (thus the Markov source is a constant source) for a sufficiently small bias value. For such a source, the terminal stage conditional measure can be made atomic via a careful construction of equilibrium policies for earlier time stages; i.e., the defined separable groups/types and discrete/quantized stage-wise equilibria through multiple stages can lead to a fully informative equilibrium for the complete game. Next, Nash equilibria of the multi-stage multi-dimensional cheap talk are analyzed. Since there may be discrete, non-discrete or even linear Nash equilibria in the single-stage multi-dimensional cheap talk by \cite[Theorem 3.4]{tacWorkArxiv}, the equilibrium policies are more difficult to characterize; however, we state the following:
\begin{thm}
	\begin{itemize}
		\item[(i)]\label{thm:multiCheap} The Nash equilibrium cannot be fully revealing in the static (single-stage) multi-dimensional cheap talk when the source has positive measure for every non-empty open set.
		\item[(ii)]\label{thm:dynamicMultiCheap}The final stage Nash equilibria cannot be fully revealing in the multi-stage multi-dimensional cheap talk for i.i.d. and Markov sources when the conditional distribution $P(\mathrm{d}\mathbf{m}_{N-1}|\mathbf{m}_{N-2})$ has positive measure for every non-empty open set.
	\end{itemize}	
\end{thm}
\begin{proof}
(i) Similar to the single-stage scalar case Theorem~\ref{noiselessDiscrete}, in an equilibrium, define two cells $C^{\alpha}$ and $C^{\beta}$, any points in those cells as $\mathbf{m}^{\alpha} \in C^{\alpha}$ and $\mathbf{m}^{\beta} \in C^{\beta}$, and the actions of the decoder as $\mathbf{u}^{\alpha}$ and $\mathbf{u}^{\beta}$ when the encoder transmits $\mathbf{m}^{\alpha}$ and $\mathbf{m}^{\beta}$, respectively. Let $F(\mathbf{m},\mathbf{u})\triangleq\|\mathbf{m}-\mathbf{u}-\vb\|^2$. Due to the equilibrium definitions from the view of the encoder; $F(\mathbf{m}^{\alpha},\mathbf{u}^{\alpha})<F(\mathbf{m}^{\alpha},\mathbf{u}^{\beta})$ and $F(\mathbf{m}^{\beta},\mathbf{u}^{\beta})<F(\mathbf{m}^{\beta},\mathbf{u}^{\alpha})$. Hence, there exists a hyperplane defined by  $F(\mathbf{z},\mathbf{u}^{\alpha})=F(\mathbf{z},\mathbf{u}^{\beta}) \Rightarrow \|(\mathbf{z}-\vb)-\mathbf{u}^{\alpha}\|^2 = \|(\mathbf{z}-\vb)-\mathbf{u}^{\beta}\|^2$. 
The hyperplane defined by the points $\mathbf{z}$ divides the space into two subspaces: let $Z^{\alpha}$ that contains $\mathbf{u}^{\alpha}$ and $Z^{\beta}$ that contains $\mathbf{u}^{\beta}$ be those subspaces. $\mathcal{C}^{\beta}$ and $Z^{\alpha}$ are disjoint subspaces since $F(\mathbf{z}+\delta (\mathbf{u}^{\beta}-\mathbf{u}^{\alpha}),\mathbf{u}^{\alpha})\geq F(\mathbf{z}+\delta (\mathbf{u}^{\beta}-\mathbf{u}^{\alpha}),\mathbf{u}^{\beta})$ for any $\delta>0$. Similarly, $\mathcal{C}^{\alpha}$ and $Z^{\beta}$ are disjoint subspaces, too. Thus, the hyperplane defined by the points $\mathbf{z}$ must lie between $\mathbf{u}^{\alpha}$ and $\mathbf{u}^{\beta}$ which implies that the length of $\vb$ along the $\mathbf{d}\triangleq\mathbf{u}^{\beta} - \mathbf{u}^{\alpha}$ direction should not exceed half of the distance between $\mathbf{u}^{\alpha}$ and $\mathbf{u}^{\beta}$; i.e., $\|\vb_{\mathbf{d}}\|\leq\|\mathbf{d}\|/2$, where $\vb_{\mathbf{d}}$ is the projection of $\vb$ along the direction of $\mathbf{d}$. Since $\mathbf{d}$ can be any vector at a fully revealing equilibrium by the assumption on the source (i.e., the source has positive measure for every non-empty open set), $\|\vb_{\mathbf{d}}\|\leq\|\mathbf{d}\|/2$ cannot be satisfied unless $\vb=\mathbf{0}$. Thus, there cannot be a fully revealing equilibrium in the static multi-dimensional cheap talk.\newline
(ii) The proof is the multi-dimensional extension of Theorem~\ref{thm:encoderNQuantized} for i.i.d. sources, and Theorem~\ref{thm:encoderNMarkovReductionQuantized} for Markov sources. \end{proof}

\subsection{Multi-Stage Cheap Talk under Stackelberg Equilibria}
Here, the cheap talk game is analyzed under the Stackelberg formulation for both scalar and multi-dimensional sources. In this case, admittedly the problem is less interesting.
\begin{thm}\label{thm:multiStack}
	An equilibrium has to be fully revealing in	the multi-stage Stackelberg cheap talk game regardless of the source model.	 
\end{thm}
\begin{proof}
The last stage decoder cost $J_{N-1}^d(\gamma_{N-1}^e,\gamma_{N-1}^d) = \mathbb{E} [\|\mathbf{m}_{N-1}-\mathbf{u}_{N-1}\|^2 | \mathcal{I}_{N-1}^d]$ is minimized by choosing the optimal action $\mathbf{u}^{*}_{N-1}=\gamma_{N-1}^{*,d}( \mathcal{I}_{N-1}^d)=\mathbb{E}[\mathbf{m}_{N-1}|\mathcal{I}_{N-1}^d]$. For the previous stage, the decoder can minimize $J_{N-2}^d(\gamma_{N-1}^{*,e},\gamma_{N-2}^e,\gamma_{N-1}^{*,d},\gamma_{N-2}^d) = \mathbb{E} [\|\mathbf{m}_{N-2}-\mathbf{u}_{N-2}\|^2 + J_{N-1}^{*,d}(\gamma_{N-1}^{*,e},\gamma_{N-1}^{*,d}) | \mathcal{I}_{N-2}^d]$ by choosing his policy as $\mathbf{u}^{*}_{N-2}=\gamma_{N-2}^{*,d}( \mathcal{I}_{N-2}^d)=\mathbb{E}[\mathbf{m}_{N-2}|\mathcal{I}_{N-2}^d]$. Similarly, the optimal decoder actions become $\mathbf{u}^{*}_{k}=\gamma_{k}^{*,d}( \mathcal{I}_{k}^d)=\mathbb{E}[\mathbf{m}_{k}|\mathcal{I}_{k}^d]=\mathbb{E}[\mathbf{m}_{k}|\mathbf{x}_{[0,k]}]$. Then, due to the Stackelberg assumption, the total encoder cost becomes $J^e(\gamma^e_{[0,N-1]},\gamma^{*,d}_{[0,N-1]}) = \mathbb{E} \left[\sum\limits_{k=0}^{N-1} \|\mathbf{m}_k-\mathbf{u}_k\|^2\right] + N \|\mathbf{b}\|^2$ by the smoothing property of the expectation. Thus, as in the static game setup \cite[Theorem 3.3]{tacWorkArxiv}, the goals of the players become essentially the same, and the result follows. \end{proof}

\section{Multi-Stage Quadratic Gaussian Signaling Games}
\label{sec:dynamicScalarSignalingGame}

The multi-stage signaling game setup is similar to the multi-stage cheap talk setup except that there exists an additive Gaussian noise channel between the encoder and the decoder at each stage, and the encoder has a {\it soft} power constraint. For the purpose of illustration, the system model of the $2$-stage signaling game is depicted in Fig.~\ref{figure:allScheme}-(\subref{figure:genScheme}). Here, source is assumed to be an $n$-dimensional Markovian source with initial Gaussian distribution; i.e., $\vMz \sim \mathcal{N}(0,\Sigma_{\vMz})$ and $\mathbf{M}_{k+1} = G \mathbf{M}_k + \mathbf{V}_k$ where $G$ is an $n \times n$ matrix ($g$ denotes the scalar case equivalent of $G$) and $\mathbf{V}_k \sim \mathcal{N}(0,\Sigma_{\mathbf{V}_k})$ is an i.i.d. Gaussian noise sequence. The channels between the encoder and the decoder are assumed to be i.i.d. additive Gaussian channels; i.e., $\mathbf{W}_k \sim \mathcal{N}(0,\Sigma_{\mathbf{W}_k})$, and $\mathbf{W}_k$ and $\mathbf{V}_l$ are independent. 
At the $k$-th stage of the $N$-stage game, the encoder knows the values of $\mathcal{I}_k^e = \{\mathbf{m}_{[0,k]}, \mathbf{y}_{[0,k-1]}\}$ (a noiseless feedback channel is assumed) and the decoder knows the values of $\mathcal{I}_k^d = \{\mathbf{y}_{[0,k]}\}$ with $\mathbf{y}_k = \mathbf{x}_k + \mathbf{w}_k$. Thus, under the policies considered, $\mathbf{x}_k=\gamma_k^e(\mathcal{I}_k^e)$ and $\mathbf{u}_k=\gamma_k^d(\mathcal{I}_k^d)$. The encoder's goal is to minimize \eqref{eq:dynamicEncCost} with $c_k^e\left(\mathbf{m}_k,\mathbf{x}_k,\mathbf{u}_k\right) = \|\mathbf{m}_k-\mathbf{u}_k-\vb\|^2 + \lambda \|\mathbf{x}_k\|^2$, whereas, the decoder's goal is to minimize \eqref{eq:dynamicDecCost} with $c_k^d\left(\mathbf{m}_k,\mathbf{u}_k\right) = \|\mathbf{m}_k-\mathbf{u}_k\|^2$, by finding the optimal policy sequences $\gamma^{*,e}_{[0,N-1]}$ and $\gamma^{*,d}_{[0,N-1]}$, respectively, where the lengths of the vectors are defined in $L_2$ norm and $\vb$ is the bias vector.
Note that a power constraint with an associated multiplier $\lambda$ is appended to the cost function of the encoder, which corresponds to power limitation for transmitters in practice. If $\lambda=0$, this corresponds to the setup with no power constraint at the encoder. 

\subsection{Nash Equilibrium Analysis of Multi-Stage Quadratic Gaussian Signaling Games\protect\footnote{\textmd{In our related paper \protect\cite{tacWorkArxiv}, for Theorem 4.1 and Section V.B, we incidentally used the information theoretic lower bounds for the Nash equilibrium analysis. However, due to the assumption on the optimal decoder action; i.e., $\vuu=\mathbb{E}[\vmm|\vyy]$, the information theoretic arguments are actually valid for the Stackelberg case (see \protect\cite{SS_PhD} for details).}}}
\label{sec:multiMultiNash}

In multi-stage quadratic Gaussian signaling games, affine policies constitute an invariant subspace under best response maps for Nash equilibria for both scalar ($n=1$) and vector ($n>1$) Gauss-Markov sources.
\begin{thm}
	\begin{itemize}
		\item [(i)]\label{thm:affineDecoderMulti} If the encoder uses affine policies at all stages, then the decoder will be affine at all stages.
		\item [(ii)] \label{thm:multiAffineEncoder}	If the decoder uses affine policies at all stages, then the encoder will be affine at all stages.		
	\end{itemize}
\end{thm}
\begin{proof}
(i) Let the encoder policies be $\mathbf{x}_k = \gamma_k^e(\mathbf{m}_{[0,k]},\mathbf{y}_{[0,k-1]}) = \sum_{i=0}^{k} A_{k,i}\,\mathbf{m}_i + \sum_{i=0}^{k-1} B_{k,i}\,\mathbf{y}_i + \mathbf{C}_k$ where $A_{k,i}$ and $B_{k,i}$ are $n\times n$ matrices, and $\mathbf{C}_k$ is $n\times1$ vector for $k\leq N-1$ and $i\leq k$. Then, the optimal decoder actions are $\mathbf{u}^{*}_{k}=\mathbb{E}[\mathbf{m}_{k}|\mathcal{I}_{k}^d]=\mathbb{E}[\mathbf{m}_{k}|\mathbf{y}_{[0,k]}]$ for $k\leq N-1$. Notice that $\mathbf{y}_{[0,k]}$ is multivariate Gaussian since $\mathbf{y}_k=\mathbf{x}_k+\mathbf{w}_k$. This proves that $\gamma_{k}^{*,d}(\mathcal{I}_{k}^d)$ is an affine function of $\mathbf{y}_{[0,k]}$ due to the joint Gaussianity of $\mathbf{m}_{k}$ and $\mathbf{y}_{[0,k]}$.\newline	
(ii) For a fixed affine decoder, the optimal encoder can be computed algebraically (for a detailed derivation, see \cite[Theorem 3.5.1]{SS_PhD}). \end{proof}

While it provides a structural result on the plausibility of affine equilibria, Theorem~\ref{thm:multiAffineEncoder} does not lead to a conclusion about the existence of an informative equilibrium. It may be tempting to apply fixed point theorems (such as Brouwer's fixed point theorem \cite{basols99}) to establish the existence of informative equilibria; however, that there always exist a non-informative equilibrium for the cheap talk game also applies to the signaling game \cite{tacWorkArxiv}. Later on, we will make information theoretic arguments (in Theorem~\ref{thm:itNStage}) for the existence of informative equilibria for the Stackelberg setup, but this is not feasible for the Nash setup. However, the informativeness analysis of the $2$-stage signaling-game can be accomplished by analyzing the fixed points of the invariant set of affine policies in Theorem~\ref{thm:multiAffineEncoder} as follows (for a proof, due to space constraints, see \cite{SS_PhD}):
\begin{thm}
	For the $2$-stage signaling game setup under affine encoder and decoder assumptions, 
	\begin{enumerate}
		\item[(i)] If $\lambda>\max\Big\{{(g^2+1)\sigma_{M_0}^2 \over \sigma_{W_0}^2}, {\sigma_{M_1}^2\over\sigma_{W_1}^2}\Big\}$, then there does not exist an informative affine equilibrium.
		\item[(ii)] If ${\sigma_{M_1}^2\over\sigma_{W_1}^2}<\lambda \leq {(g^2+1)\sigma_{M_0}^2 \over \sigma_{W_0}^2}$, then the second stage message $m_1$ is not used in the game.
		\item[(iii)] If ${(g^2+1)\sigma_{M_0}^2 \over \sigma_{W_0}^2}<\lambda \leq {\sigma_{M_1}^2\over\sigma_{W_1}^2}$, the equilibrium is informative if and only if $\sigma_{M_1}^2\geq4b^2$ and $\max\left\{{\sigma_{M_1}^2-2b^2-\sqrt{\sigma_{M_1}^2}\sqrt{\sigma_{M_1}^2-4b^2}\over2\sigma_{W_1}^2}, \left(g^2+1\right){\sigma_{M_0}^2 \over \sigma_{W_0}^2}\right\} < \lambda<{\sigma_{M_1}^2-2b^2+\sqrt{\sigma_{M_1}^2}\sqrt{\sigma_{M_1}^2-4b^2}\over2\sigma_{W_1}^2}$.
	\end{enumerate}
\end{thm}
The above analysis can be carried over to the $N$-stage signaling game; however, this would involve $(3N^2+5N)/2$ equations and as many unknowns.

\subsection{Stackelberg Equilibrium Analysis of Multi-Stage Quadratic Gaussian Signaling Games}

Here, the signaling game is analyzed under the Stackelberg concept. 

\subsubsection{Multi-Stage Stackelberg Equilibria for Scalar Gauss-Markov Sources}
\label{sec:stackelbergSignalScalar}
The conditions for informative Stackelberg equilibria with scalar sources are characterized below. The proof is in Appendix~\ref{appendixProofStackelbergSingaling}.

\begin{thm}\label{thm:itNStage}
	If \mbox{ $\lambda\geq\max_{k\leq N-1}{\sigma_{M_k}^2\over\sigma_{W_k}^2}\sum_{i=0}^{N-k-1} g^{2i}$}, there does not exist an informative (affine or non-linear) equilibrium in the $N$-stage scalar signaling game under the Stackelberg assumption; i.e., the only equilibrium is the non-informative one. Otherwise, an equilibrium has to be always linear.
\end{thm}
Now consider the multi-stage Stackelberg signaling game with a discounted infinite horizon and a discount factor $\beta\in\left(0,1\right)$; i.e., $J^e(\gamma^e,\gamma^d) = \mathbb{E}\left[\sum_{i=0}^{\infty} \beta^i \left(\left(m_i-u_i-b\right)^2 + \lambda x_i^2\right)\right]$ and $J^d(\gamma^e,\gamma^d) = \mathbb{E}\left[\sum_{i=0}^{\infty} \beta^i \left(m_i-u_i\right)^2\right]$. The proof is in Appendix~\ref{appendixProofStackelbergSingalingInf}.
\begin{cor}\label{thm:itInfStage}
	If \mbox{ $\lambda\geq\max_{k=0,1,\ldots}{\sigma_{M_k}^2\over\sigma_{W_k}^2}{1\over1-\beta g^2}$} where $\beta g^2<1$, there does not exist an informative (affine or non-linear) equilibrium in the infinite horizon discounted multi-stage Stackelberg signaling game for scalar Gauss-Markov sources.	
\end{cor}

\subsubsection{Multi-Stage Stackelberg Equilibria for Vector Gauss-Markov Sources}
\label{sec:stackelbergSignalMulti}
Linear policies are optimal for scalar sources as shown in Section~\ref{sec:stackelbergSignalScalar}. Before analyzing the multi-dimensional setup, it will be appropriate to review the optimality of linear policies in Gaussian setups for the classical communication theoretic setup when the bias term is absent: Optimality of linear coding policies for scalar Gaussian source-channel pairs with noiseless feedback has been known since 1960s, see e.g. \cite{Goblick}. Optimal linear encoders for single-stage setups have been studied in \cite{bas80,TangulBasaretal}. When the source and the channel are multi-dimensional, linear policies may not be optimal; see \cite{Pilc}, \cite[Chapter 11]{YukselBasarBook} and \cite{ZaidiChapter} for a detailed discussion and literature review. It is evident from Theorem \ref{thm:multiAffineEncoder} that when the encoder is linear, the optimal decoder is linear. In this case, a relevant problem is to find the optimal Stackelberg policy among the linear or affine class. In the following, a dynamic programming approach is adapted to find such Stackelberg equilibria. Building on the optimality of linear innovation encoders, we restrict the analysis to such encoders; i.e., we consider a sub-optimal scenario. Our analysis builds on and generalizes the arguments in \cite[Theorem 3]{ZaidiChapter} and \cite{basban94}. The proof is in Appendix~\ref{appendixProofMultiStackelberg}.
\begin{thm}\label{thm:multiStackelberg}
	Suppose that $G$, $\Sigma_{\vMz}$ and $\Sigma_{\vvv_k}$ are diagonal. Suppose further that the innovation is given by $\vmt_k \triangleq \vmm_k-\mathbb{E}[\vmm_k|\vyy_{[0,k-1]}]$ with $\vmt_0 = \vmm_0$, and that the encoder linearly encodes the innovation. Then, an optimal such linear policy can be computed through dynamic programming with value functions $V_k\left(\Sigma_{\vmt_k}\right)\triangleq\tr\left(K_k \Sigma_{\vmt_k} + L_k\right)$ that satisfy the terminal condition $V_N\left(\Sigma_{\vmt_N}\right)=0$ with diagonal $K_k$ matrices for $k\leq N-1$. 	
\end{thm}

\section{Concluding Remarks}
In this paper, we studied Nash and Stackelberg equilibria for multi-stage quadratic cheap talk and signaling games. We established qualitative (e.g. on full revelation, quantization nature, linearity, informativeness and non-informativeness) and quantitative properties (on linearity or explicit computation) of Nash and Stackelberg equilibria under misaligned objectives.

\appendix

\section{Proof of Theorem~\ref{thm:itNStage}} \label{appendixProofStackelbergSingaling}
Similar to that in Theorem~\ref{thm:multiStack}, the optimal decoder actions are  $u^{*}_{k}=\gamma_{k}^{*,d}( \mathcal{I}_{k}^d)=\mathbb{E}[m_{k}|\mathcal{I}_{k}^d]=\mathbb{E}[m_{k}|y_{[0,k]}]$, and the total encoder cost becomes $J^e(\gamma^e_{[0,N-1]},\gamma^d_{[0,N-1]}) = \mathbb{E} \left[\sum\limits_{k=0}^{N-1} \mathbb{E}[(m_k-\mathbb{E}[m_{k}|\mathcal{I}_{k}^d])^2+b^2+ \lambda x_k^2|\mathcal{I}_k^d]\right]$. 
This is an instance of problems studied in \cite{simultaneousDesign}, and can be reduced to a team problem where both the players are minimizing the same cost. The linearity of the optimal policies can be deduced from \cite{simultaneousDesign}. Here, we adapt the proof in \cite{simultaneousDesign} to our setup. From the chain rule, $I(m_k;y_{[0,k]})=I(m_k;y_{[0,k-1]})+I(m_k;y_k|y_{[0,k-1]})$. By following similar arguments to those in \cite{simultaneousDesign} and \cite[Theorem 11.3.1]{YukselBasarBook},
\begin{align*}
& I(m_k;y_k|y_{[0,k-1]}) = h(y_k|y_{[0,k-1]})-h(y_k|m_k,y_{[0,k-1]}) \\
&\quad\leq h(y_k)-h(y_k|\gamma_k^e(m_k,y_{[0,k-1]})) \\
&\quad = I\left(\gamma_k^e(m_k,y_{[0,k-1]});y_k\right) = I(x_k;y_k) \leq \sup I(x_k;y_k) \\ 
&\quad= {1\over2}\log_2\left(1+{P_k\over\sigma_{W_k}^2}\right)\triangleq \widehat{C}_k \text{ where } P_k = \mathbb{E}[x_k^2]\,.
\end{align*}
It can be seen that $m_k-\mathbb{E}[m_k|m_{k-1}]$ is orthogonal to the random variables $m_{k-1}, y_{[0,k-1]}$ where $y_{[0,k-1]}$ is included due to the Markov chain $m_k\leftrightarrow m_{k-1} \leftrightarrow (y_{[0,k-1]})$. By using this orthogonality, it follows that
\begin{align}
&\mathbb{E}[(m_k-\mathbb{E}[m_k|y_{[0,k-1]}])^2] = \mathbb{E}[(m_k-\mathbb{E}[m_k|m_{k-1}])^2] \nn\\
&\qquad\qquad\qquad\qquad + \mathbb{E}[(\mathbb{E}[m_k|m_{k-1}]-\mathbb{E}[m_k|y_{[0,k-1]}])^2] \nn\\
&\qquad\qquad\overset{(a)}{=} \mathbb{E}[(m_k-\mathbb{E}[m_k|m_{k-1}])^2] + \mathbb{E}\Big[\big(\mathbb{E}[m_k|m_{k-1}]\nn\\
&\qquad\qquad\qquad\qquad-\mathbb{E}[\mathbb{E}[m_k|m_{k-1}]|y_{[0,k-1]}]\big)^2\Big] \nn\\
&\qquad\qquad\overset{(b)}{=} \sigma_{V_{k-1}}^2 + g^2\mathbb{E}[(m_{k-1}-\mathbb{E}[m_{k-1}|y_{[0,k-1]}])^2] \nn\\
&\qquad\qquad\overset{(c)}{\geq} \sigma_{V_{k-1}}^2 +g^2\sigma_{M_{k-1}}^2\,2^{-2C_{k-1}}\,,
\label{eq:mkyk1}
\end{align}
where $C_k\triangleq\sup I(m_k;y_{[0,k]})$. Here, (a) holds due to the iterated expectation rule and the Markov chain property, 
(b) holds since $\mathbb{E}[m_k|m_{k-1}]=\mathbb{E}[gm_{k-1}+v_{k-1}|m_{k-1}]=gm_{k-1}$, and (c) holds due to \cite[Lemma 11.3.1]{YukselBasarBook}. From \cite[Lemma 11.3.2]{YukselBasarBook},
$I(m_k;y_{[0,k-1]})$ is maximized with linear policies, and the lower bound of \eqref{eq:mkyk1}, $\mathbb{E}[(m_k-\mathbb{E}[m_k|y_{[0,k-1]}])^2] \geq \sigma_{V_{k-1}}^2 +g^2\sigma_{M_{k-1}}^2\,2^{-2C_{k-1}}\triangleq\sigma_{M_k}^2\,2^{-2 \widetilde{C}_k}$, is achievable through linear policies where $\sup I(m_k;y_{[0,k-1]})\triangleq\widetilde{C}_k={1\over2}\log_2\left(\sigma_{M_k}^2\over\sigma_{V_{k-1}}^2 +g^2\sigma_{M_{k-1}}^2 2^{-2C_{k-1}}\right)$. Thus, we have the following recursion on upper bounds on mutual information for the $N$-stage signaling game:
\begin{align*}
&C_k=\sup I(m_k;y_{[0,k]})=\widetilde{C}_k+\widehat{C}_k\nn\\
&={1\over2}\log_2\left(\sigma_{M_k}^2\over\sigma_{V_{k-1}}^2 +g^2\sigma_{M_{k-1}}^2 2^{-2C_{k-1}}\right)+{1\over2}\log_2\left(1+{P_k\over\sigma_{W_k}^2}\right)
\end{align*}
for $k\leq N-1$ with $C_0 = {1\over2}\log_2\left(1+{P_0\over\sigma_{W_0}^2}\right)$. Let the lower bound on $\mathbb{E} \left[\left(m_k-\mathbb{E}[m_k|y_{[0,k]}]\right)^2\right]$ be $\Delta_k$; i.e., $\mathbb{E} \left[\left(m_k-\mathbb{E}[m_k|y_{[0,k]}]\right)^2\right] \geq \sigma_{M_k}^2\,2^{-2 C_k}\triangleq \Delta_k$. Then the following recursion holds for the $N$-stage signaling game:
\begin{align}
\Delta_k &= {\sigma_{V_{k-1}}^2 +g^2 \Delta_{k-1}\over1+{P_k\over\sigma_{W_k}^2}} 
\text{ for } k=1,2,\ldots,N-1\nn
\end{align}
with $\Delta_0 = {\sigma_{M_0}^2\over1+{P_0\over\sigma_{W_0}^2}}$. Since $\Delta_k=\sigma_{M_k}^2\,2^{-2 C_k}$ by definition, $\Delta_k\leq\sigma_{M_k}^2$ for $k\leq N-1$.
In an equilibrium, since the decoder always chooses $u_k=\mathbb{E}[m_k|y_{[0,k]}]$ for $k\leq N-1$, the total encoder cost for the first stage can be lower bounded by $J_0^{e,lower} = \sum_{i=0}^{N-1} \left(\Delta_i+\lambda P_i + b^2\right)$. Now observe the following:
\begin{align*}
{\partial \Delta_l \over\partial P_k} = \begin{cases}
0 &\mbox{if } l<k \\
\begin{aligned}[c]
&g^2\left(1+{P_l\over\sigma_{W_l}^2}\right)^{-1}{\partial \Delta_{l-1} \over\partial P_k}-{1\over\sigma_{W_l}^2}{\partial P_l\over\partial P_k}\\
&\times\left(\sigma_{V_{l-1}}^2 +g^2 \Delta_{l-1}\right)\left(1+{P_l\over\sigma_{W_l}^2}\right)^{-2}
\end{aligned} &\mbox{if } l\geq k 
\end{cases}\,,
\end{align*}
where ${\partial P_l\over\partial P_k}=0$ for $l<k$ due to the information structure of the encoder. Then, ${\partial J_0^{e,lower} \over\partial P_{N-1}}\geq\lambda-{\sigma_{M_{N-1}}^2\over\sigma_{W_{N-1}}^2}$ can be obtained. 
If $\lambda>{\sigma_{M_{N-1}}^2\over\sigma_{W_{N-1}}^2}$, then ${\partial J_0^{e,lower} \over\partial P_{N-1}}>0$, which implies that $J_0^{e,lower}$ is an increasing function of $P_{N-1}$. For this case, in order to minimize $J_0^{e,lower}$, $P_{N-1}$ must be chosen as $0$; i.e., $P_{N-1}^*=0$. Then, for $\lambda>{\sigma_{M_{N-1}}^2\over\sigma_{W_{N-1}}^2}$, 
\begin{align*}
&{\partial J_0^{e,lower} \over\partial P_{N-2}} = \lambda\left(1+{\partial P_{N-1}\over\partial P_{N-2}}\right)+\sum_{i=N-2}^{N-1} {\partial \Delta_i \over\partial P_{N-2}}\nn\\
&=\lambda\left(1+{\partial P_{N-1}\over\partial P_{N-2}}\right)+\left(g^2\left(1+{P_{N-1}\over\sigma_{W_{N-1}}^2}\right)^{-1}+1\right){\partial \Delta_{N-2} \over\partial P_{N-2}}\nn\\
&\quad-\left(\sigma_{V_{N-2}}^2 +g^2 \Delta_{N-2}\right)\left(1+{P_{N-1}\over\sigma_{W_{N-1}}^2}\right)^{-2}{1\over\sigma_{W_{N-1}}^2}{\partial P_{N-1}\over\partial P_{N-2}} \nn\\
&\stackrel{(a)}{=} \lambda+{\partial \Delta_{N-2} \over\partial P_{N-2}}\left(g^2+1\right)\geq\lambda-{\sigma_{M_{N-2}}^2\over\sigma_{W_{N-2}}^2}\left(g^2+1\right)\,.
\end{align*}
Here, (a) holds since $P_{N-1}^*=0$ for $\lambda>{\sigma_{M_{N-1}}^2\over\sigma_{W_{N-1}}^2}$. If $\lambda>\max\Big\{{\sigma_{M_{N-1}}^2\over\sigma_{W_{N-1}}^2},{\sigma_{M_{N-2}}^2\over\sigma_{W_{N-2}}^2}\left(g^2+1\right)\Big\}$, then ${\partial J_0^{e,lower} \over\partial P_{N-2}}>0$, which implies that $J_0^{e,lower}$ is an increasing function of $P_{N-2}$. For this case, in order to minimize $J_0^{e,lower}$, $P_{N-2}$ must be chosen as $0$. By following the similar approach and assumptions on $\lambda$, since $P_{N-1}^*=P_{N-2}^*=\cdots=P_{k+1}^*=0$, we have the following: 
\begin{align*}
&{\partial J_0^{e,lower} \over\partial P_k} = \lambda+\sum_{i=k}^{N-1} {\partial \Delta_i \over\partial P_k}=\lambda+{\partial \Delta_k \over\partial P_k}\sum_{i=k}^{N-1}\prod_{j=k+1}^{i}g^2 \nn\\
&=\lambda-{\sigma_{V_{k-1}}^2 +g^2 \Delta_{k-1}\over\sigma_{W_k}^2}\left(1+{P_k\over\sigma_{W_k}^2}\right)^{-2}\sum_{i=k}^{N-1}\prod_{j=k+1}^{i}g^2\nn\\
&\geq\lambda-{\sigma_{M_k}^2\over\sigma_{W_k}^2}\sum_{i=0}^{N-k-1} g^{2i}\,,
\end{align*}
where $\prod_{i=k}^{l}=1$ if $k>l$. If $\lambda>{\sigma_{M_k}^2\over\sigma_{W_k}^2}\sum_{i=0}^{N-k-1} g^{2i}$, then ${\partial J_0^{e,lower}\over\partial P_k}>0$, which implies that $J_0^{e,lower}$ is an increasing function of $P_k$. For this case, in order to minimize $J_0^{e,lower}$, $P_k$ must be chosen as $0$. 

By combining all the results above, it can be deduced that if $\lambda>\max_{k\leq N-1}{\sigma_{M_k}^2\over\sigma_{W_k}^2}\sum_{i=0}^{N-k-1} g^{2i}$, the lower bound $J_0^{e,lower}$ of the encoder costs $J_0^e$ is minimized by choosing $P_0^*=P_1^*=\cdots=P_{N-1}^*=0$; that is, the encoder does not signal any output. Hence, the encoder engages in a non-informative equilibrium and the minimum cost becomes $J_0^e = J_0^{e,lower} = \left(\sum_{i=0}^{N-1}\sigma_{M_i}^2\right)+Nb^2$.  \hspace*{\fill}\qed

\section{Proof of Corollary~\ref{thm:itInfStage}} \label{appendixProofStackelbergSingalingInf}
For the infinite horizon case, it can be observed
\begin{align}
\inf_{\gamma^e} &\limsup_{N\rightarrow\infty} J^e(\gamma^e_{[0,N-1]},\gamma^d_{[0,N-1]}) \nn\\
&\geq \limsup_{N\rightarrow\infty} \inf_{\gamma^e_{[0,N-1]}} \sum_{i=0}^{N-1} \beta^i \left(\Delta_i+\lambda P_i + b^2\right)\nn\,.
\end{align}
Thus, $\limsup_{N\rightarrow\infty} \inf_{\gamma^e_{[0,N-1]}} \sum_{i=0}^{N-1} \beta^i \left(\Delta_i+\lambda P_i + b^2\right)$ is achieved at a non-informative equilibrium if $\lambda>\limsup_{N\rightarrow\infty}\max_{k\leq N-1}{\sigma_{M_k}^2\over\sigma_{W_k}^2}\sum_{i=0}^{N-k-1}\beta^i g^{2i}={\sigma_{M_k}^2\over\sigma_{W_k}^2}{1\over1-\beta g^2}$ for $\beta g^2<1$. Hence, if $\lambda\geq\max_{k=0,1,\ldots}{\sigma_{M_k}^2\over\sigma_{W_k}^2}{1\over1-\beta g^2}$, then the lower bound $J_0^{e,lower}$ of the encoder costs $J_0^e$ is minimized by choosing $P_0=P_1=\cdots=0$, and the minimum cost becomes $J_0^e = J_0^{e,lower} = \sum_{i=0}^{\infty}\beta^i\left(\sigma_{M_i}^2+b^2\right)$ at this non-informative equilibrium. 
\hspace*{\fill}\qed

\section{Proof of Theorem~\ref{thm:multiStackelberg}} \label{appendixProofMultiStackelberg}
We will follow an approach similar to that in \cite{ZaidiChapter} which restricted the analysis to a team problem and a scalar channel; \cite{ZaidiChapter} in turn builds on \cite{basban94}, which considers continuous time systems. Since the ($k+1$)st stage encoder policy only transmits the linearly encoded innovation by assumption, $\vxx_k=\gamma_k^e(\mathcal{I}_k^e)=A_k \vmt_k$ where $A_k$ is an $n \times n$ matrix for $k\leq N-1$. Then the decoder receives $\vyy_k=\vxx_k+\vww_k=A_k \vmt_k+\vww_k$ and applies the action $\vuu_k=\gamma_k^d(\mathcal{I}_k^d)=\mathbb{E}[\vmm_k|\vyy_{[0,k]}]$ to minimize his stage-wise cost $\|\ver_k\|^2 \triangleq \mathbb{E}[\|\vmm_k-\vuu_k\|^2]=\mathbb{E}[(\vmm_k-\vuu_k)^T(\vmm_k-\vuu_k)]=\tr\left(\Sigma_{\ver_k}\right)$ for $k\leq N-1$ where $\Sigma_{\mathbf{R}}$ stands for the covariance matrix of the random variable ${\mathbf{R}}$; i.e., $\Sigma_{\mathbf{R}} \triangleq \mathbb{E}[(\mathbf{R}-\mathbb{E}[\mathbf{R}])(\mathbf{R}-\mathbb{E}[\mathbf{R}])^T]$. Due to the orthogonality of $\vmt_k$ and $\vyy_{[0,k-1]}$, and the iterated expectations rule,  $\vuu_k=\mathbb{E}[\vmm_k|\vyy_{[0,k]}]=\mathbb{E}\left[\vmt_k+\mathbb{E}[\vmm_k|\vyy_{[0,k-1]}]|\vyy_{[0,k]}\right]=\mathbb{E}[\vmt_k|\vyy_k]+\mathbb{E}[\vmm_k|\vyy_{[0,k-1]}]$, and it follows that $\ver_k=\vmm_k-\vuu_k=\vmm_k-\mathbb{E}[\vmt_k|\vyy_k]-\mathbb{E}[\vmm_k|\vyy_{[0,k-1]}]=\vmt_k-\mathbb{E}[\vmt_k|\vyy_k]$. Since $\mathbb{E}[\vmt_k|\vyy_k]=\Sigma_{\vmt_k} A_k^T \left(\Sigma_{\vyy_k}\right)^{-1} \vyy_k$, the stage-wise cost of the decoder becomes the trace of the following:
\begin{align}
\Sigma_{\ver_k} &= 
\Sigma_{\vmt_k}-\Sigma_{\vmt_k} A_k^T \left(A_k\Sigma_{\vmt_k} A_k^T + \Sigma_{\vww_k}\right)^{-1} A_k \Sigma_{\vmt_k} \nn\\
&\overset{(a)}{=} \Sigma_{\vmt_k}^{1/2}\left(I+H_k^TH_k\right)^{-1}\Sigma_{\vmt_k}^{1/2}\,,
\label{eq:sigmaError}
\end{align}
where (a) follows by utilizing the matrix inversion lemma, \mbox{$(I+UWV)^{-1}=I-U(W^{-1}+VU)^{-1}V$}, where $U=H_k^T$, $W=I$, $V=H_k$, and $H_k\triangleq\Sigma_{\vww_k}^{-1/2} A_k \Sigma_{\vmt_k}^{1/2}$. Since $\mathbb{E}[\vmm_k|\vyy_{[0,k-1]}]$ and $\vmt_k$ are orthogonal, $\mathbb{E}[\vmm_k|\vyy_k]=\Sigma_{\vmt_k} A_k^T\left(\Sigma_{\vyy_k}\right)^{-1} \vyy_k$.
Then the innovation and its covariance matrix can be expressed recursively as follows:
\begin{align}
\vmt_{k+1} 
&= G\vmm_k+\vvv_k - \mathbb{E}[\vmm_{k+1}|\vyy_{[0,k-1]}] - \mathbb{E}[\vmm_{k+1}|\vyy_k] \nn\\
&= G\vmt_k+\vvv_k - G \Sigma_{\vmt_k} A_k^T \left(\Sigma_{\vyy_k}\right)^{-1} \vyy_k \,,\nn\\
\Sigma_{\vmt_{k+1}}&=G\Sigma_{\vmt_k}^{1/2}\left(I+H_k^TH_k\right)^{-1}\Sigma_{\vmt_k}^{1/2}G^T+\Sigma_{\vvv_k}  \,.
\label{eq:sigmaMkRecursion}
\end{align}
Further, the stage-wise cost of the encoder is 
\begin{align}
\mathbb{E}&\left[\|\vmm_k-\vuu_k-\vb\|^2+\lambda\|\vxx_k\|^2\right] = \tr\left(\Sigma_{\ver_k}\right) + \tr\left(\lambda\Sigma_{\vxx_k}\right) + \|\vb\|^2\nn\\
&=\tr\left(\Sigma_{\vmt_k}\left(I+H_k^TH_k\right)^{-1}\right) + \tr\left(\lambda H_k^T \Sigma_{\vww_k}H_k\right)+ \|\vb\|^2\,.
\label{eq:traceEncoder} 
\end{align}
Let the value functions be $V_k\left(\Sigma_{\vmt_k}\right)=\tr\left(K_k \Sigma_{\vmt_k} + L_k\right)$ with $K_k$ being diagonal. In the following we show that there exist such $V_k$ that satisfy Bellman's principle of optimality \cite[Theorem 3.2.1]{HernandezLermaMCP}. Here, $V_k\left(\Sigma_{\vmt_k}\right) \triangleq \min_{H_k} \Bigg(\mathcal{C}_k\left(\Sigma_{\vmt_k},H_k\right)+V_{k+1}\left(\Sigma_{\vmt_{k+1}}\right)\Bigg)$, and $\mathcal{C}_k\left(\Sigma_{\vmt_k},H_k\right)\triangleq \tr\left(\Sigma_{\ver_k}\right) + \tr\left(\lambda\Sigma_{\vxx_k}\right) +\|\vb\|^2$ is the stage-wise cost of the $k$-th stage encoder. Then, since $K_{k+1}$ and $L_{k+1}$ do not depend on $H_k$,
\begin{align}
&V_k\left(\Sigma_{\vmt_k}\right) = \min_{H_k} \Bigg(\mathcal{C}_k\left(\Sigma_{\vmt_k},H_k\right)+V_{k+1}\left(\Sigma_{\vmt_{k+1}}\right)\Bigg) \nn\\
&\overset{(a)}{=} \tr\left(K_{k+1} \Sigma_{\vvv_k} + L_{k+1}\right)+\|\vb\|^2 + \min_{H_k} \Bigg(\underbrace{\tr\left(\lambda H_k^T \Sigma_{\vww_k}H_k\right)}_{\triangleq\mathfrak{C}}\nn\\
&+\underbrace{\tr\Big(\Sigma_{\vmt_k}^{1/2}\left(G^TK_{k+1} G+I\right)\Sigma_{\vmt_k}^{1/2}\left(I+H_k^TH_k\right)^{-1}\Big)}_{\triangleq\mathfrak{P}}\Bigg) \;,
\label{eq:valueEncoder}
\end{align}
where (a) follows by substituting $\mathcal{C}_k\left(\Sigma_{\vmt_k},H_k\right)$ using \eqref{eq:traceEncoder} and employing \eqref{eq:sigmaMkRecursion}. The equivalent problem of the minimization of $\mathfrak{P}$ over $H_k$ under the constraint $\mathfrak{C}=\mu_k$ is considered in \cite{bas80}, and for every $\mu_k=\mathfrak{C}$, the optimal $H_k$ is found as $H_k^*=\Pi_k \zeta_k P_k^T$, where $\Pi_k$ is a unitary matrix such that $\Pi_k^T \left(\lambda\Sigma_{\vww_k}\right) \Pi_k = \text{ diag} \left(\tau_{k_1},\tau_{k_2},\ldots,\tau_{k_n}\right)\triangleq\widetilde{\Pi}_k$, $P_k$ is a unitary matrix such that $P_k^T\left(\Sigma_{\vmt_k}^{1/2}\left(G^TK_{k+1} G+I\right)\Sigma_{\vmt_k}^{1/2}\right)P_k=\text{ diag} \left(\nu_{k_1},\nu_{k_2},\ldots,\nu_{k_n}\right)$, and $\zeta_k$ is another diagonal matrix. Then the recursion of the innovation's covariance matrix \eqref{eq:sigmaMkRecursion} can be expressed as
\begin{align}
\Sigma_{\vmt_{k+1}}
&= G\Sigma_{\vmt_k}^{1/2}\left(I+P_k\zeta_k^T\zeta_k P_k^T\right)^{-1}\Sigma_{\vmt_k}^{1/2}G^T+\Sigma_{\vvv_k} \;.
\label{eq:innovRecur}
\end{align}
Then, by utilizing $H_k^*=\Pi_k \zeta_k P_k^T$, $\Pi_k^T\Pi_k=I$, $\widetilde{\Pi}_k=\Pi_k^T \left(\lambda\Sigma_{\vww_k}\right) \Pi_k$, and $P_k^TP_k=I$ in \eqref{eq:valueEncoder}, 
\begin{align}
&V_k\left(\Sigma_{\vmt_k}\right)\overset{(a)}{=}  \tr\left(K_{k+1} \Sigma_{\vvv_k} + L_{k+1}+\vb\vb^T\right) + \tr\left(\zeta_k^T \widetilde{\Pi}_k \zeta_k\right) \nn\\
&+\tr\Bigg(\Big(G^TK_{k+1} G+I\Big)\Big(I-\zeta_k^T\left(I+\zeta_k \zeta_k^T\right)^{-1}\zeta_k\Big)\Sigma_{\vmt_k}\Bigg) \,, \label{eq:valueFunctionIteration}
\end{align}
where (a) follows from the matrix inversion lemma by choosing $U=P_k\zeta_k^T$, $W=I$, and $V=\zeta_k P_k^T$ in \mbox{$(I+UWV)^{-1}=I-U(W^{-1}+VU)^{-1}V$}, and the diagonality of $\Sigma_{\vmt_k}$, $P_k$ and $\zeta_k$: Since $G$, $\Sigma_{\vmt_0}$, $K_k$ and $\Sigma_{\vvv_k}$ are diagonal for $k\leq N-1$, it is always possible to find a unitary diagonal $P_0$ such that $P_0^T\left(\Sigma_{\vmt_0}^{1/2}\left(G^TK_1 G+I\right)\Sigma_{\vmt_0}^{1/2}\right)P_0=\text{ diag} \left(\nu_{0_1},\nu_{0_2},\ldots,\nu_{0_n}\right)$, which makes $\Sigma_{\vmt_1}$ diagonal by \eqref{eq:innovRecur}. Similarly, $\Sigma_{\vmt_k}$ and $P_k$ are diagonal for $k\leq N-1$. In order to satisfy \eqref{eq:valueFunctionIteration}, since $V_N\left(\Sigma_{\vmt_N}\right)=0$, we choose $K_N=L_N=0$. Then, for $k\leq N-1$, $\{K_{k+1}, L_{k+1}\}$ is chosen according to
\begin{align}
K_k &= \Big(G^TK_{k+1} G+I\Big)\left(I-\zeta_k^T\left(I+\zeta_k \zeta_k^T\right)^{-1}\zeta_k\right) \;,\nn\\
L_k &= K_{k+1} \Sigma_{\vvv_k} + L_{k+1}+\zeta_k^T \widetilde{\Pi}_k \zeta_k+\vb\vb^T \;.
\label{eq:backwardIteration}
\end{align}
Now we verify that the diagonal $K_k$ matrices satisfy the dynamic programming recursion.

When the channel is scalar, for $k\leq N-1$, 
\begin{align*}
\begin{split}
K_k &= \Big(G^TK_{k+1} G+I\Big)\times\text{ diag}\left({\lambda\sigma^2_{W_k}\over 1+\lambda\sigma^2_{W_k}},1,1,\ldots,1\right) \,,\\
L_k &= K_{k+1} \Sigma_{\vvv_k} + L_{k+1}+\text{ diag}\left(1,0,0,\ldots,0\right)+\vb\vb^T \,. 
\end{split}
\end{align*}
The optimal linear encoder policy is  $A_k^*=\Sigma_{\vww_k}^{1/2} \zeta_k P_k^T \Sigma_{\vmt_k}^{-1/2}$ since $\Pi_k=1$ and \mbox{$\zeta_k=\left[{1\over \sqrt{\lambda\sigma^2_{W_k}}},0,\ldots,0\right]$}. \hspace*{\fill}\qed
\bibliographystyle{IEEEtran}        
\bibliography{SerkanBibliography}

\end{document}